\documentclass[letterpaper,reqno,11pt]{amsart}

\usepackage{preamble}
\usepackage{mathtools}

\newcommand{\bfemph}[1]{\textbf{\emph{#1}}}
\title{Binary X-rays of doubly stochastic matrices}
\author{Bangzheng Li \and Yuewei Liu \and Yuan Yao}
\thanks{Contact: \href{mailto:liben@mit.edu,llliliue@mit.edu,yyao1@mit.edu}{\{liben,llliliue,yyao1\}@mit.edu}}

\begin{document}

\begin{abstract}
    The X-ray of a permutation is a sequence of sums along each diagonal of the associated permutation matrix. They satisfy certain necessary constraints on distribution of the values, which are conjectured to be sufficient when the sequence is binary. By re-expressing the constraints in a form that allows for real-valued relaxations, we prove that these binary sequences are always X-rays of doubly stochastic matrices. We also disprove a natural generalization of the conjecture.
\end{abstract}

\maketitle

\section{Introduction}

We begin with an elementary question:

\begin{question}\label{question}
    Which vectors in $\mathbb{Z}^n$ can be written as the sum of two permutations $\sigma$ and $\tau$ of $[n] = \{1, 2, \dots, n\}$? (Here, each permutation is written as a vector in one-line notation.)
\end{question}

Observe that the question remains equivalent if we replace ``sum'' by ``difference'' (as $v$ can be written as a sum if and only if $v - (n+1, n+1, \dots, n+1)$ can be written as a difference). Moreover, by permuting the entries of the vector if necessary, we can always assume that $\tau$ is the identity permutation. As a result, this problem reduces to one about finding permutation matrices with prescribed sums along each diagonal:

\begin{definition}
    Given a permutation $\sigma\in S_n$, define its \bfemph{S-vector} $s(\sigma)$ to be $(\sigma(1)-1, \dots, \sigma(n)-n)\in \mathbb{Z}^{n}$ and its \bfemph{sorted S-vector} $\bar{s}(\sigma)$ to be the entries of $s(\sigma)$ sorted in non-decreasing order.
    
    Given an $n\times n$ matrix $A$ of real numbers, define its \bfemph{(diagonal) X-ray}\footnote{Also called \emph{Toeplitz X-ray} in \cite{BF2014}.} $x(A)$ to be a vector $(x_{-n+1}, \dots, x_0, \dots, x_{n-1})\in \mathbb{R}^{2n-1}$ such that $x_d = \sum_{\substack{i, j\in [n]\\ i-j=d}} a_{i,j}$ for $d \in [-n+1, n-1]$. 
    
    When $A = A_\sigma$ is a permutation matrix for $\sigma\in S_n$, we let $x(\sigma) = x(A_\sigma)\in \mathbb{N}^{2n-1}$ be the \bfemph{X-ray of permutation $\sigma$}. Note that $x(\sigma)$ is also the \bfemph{frequency vector} of $s(\sigma)$, where $x_d$ counts the number of occurrences of $d$ in $s(\sigma)$ for each $d$, and encodes the same information as $\bar{s}(\sigma)$.
\end{definition}

In addition to being an interesting object on its own, sums/differences/X-rays of permutations have been studied in relation to discrete tomography (\cite{HK1999}), and can also be seen as a special case of Skolem sequences from combinatorial design theory (\cite{CD2006}) and related to a special case of the PRV conjecture in representation theory (\cite{PRV1967}). 

Some easy necessary conditions for S-vectors are known:

\begin{proposition}[\cite{BMPS2005}, Proposition 4.1]
    The S-vector $s = (s_1, \dots, s_n)$ of a permutation satisfies $\left|\sum_{i\in I} s_i\right| \leq |I|(n-|I|)$ for all $I\subseteq [n]$.
\end{proposition}

\begin{remark}
    If we take $I = [n]$ in the above inequalities, we get an obvious equality $s_1 + \dots + s_n = 0$. This set of inequalities will be called the \bfemph{permutohedral constraints}, as they define a polytope in $\mathbb{R}^n$ that is a translated copy of a dilated permutohedron $2\Pi_n$. 
    
    The set of all integer vectors satisfying the permutohedral constraints is also the set of all possible score sequences of a round-robin tournament with $n$ players, where the score for each player is $(\#\text{ Wins}) - (\#\text{ Losses})$ but draws (worth 0 points) are possible (also see OEIS sequence A007747\footnote{The OEIS sequence counts \emph{sorted} score vectors, but the constraints are essentially equivalent.}). Hence we refer to these sequences as \bfemph{score vectors}.
\end{remark}

Frequency vectors of score vectors are not always realized as the X-ray of a permutation; we will verify in \cref{ex:counter_example} that $(-2, -2, -2, 2, 2, 2)$ is one such counter-example. Worse yet, it is known that determining whether such vectors are X-rays of a permutation, or equivalently the original \cref{question}, is NP-complete (\cite{BDGV2008}, Theorem 2), meaning that one generally do not expect there to be a small set of conditions that characterize all possible X-rays.
However, it is conjectured that all \emph{binary} frequency vectors can be realized.

\begin{conjecture}[\cite{BMPS2005}, Conjecture 4.2]\label{conj:binary}
    If a score vector contains no duplicate values, then its frequency vector can always be realized as the X-ray of a permutation.
\end{conjecture}

\begin{remark}
    It has been noted in \cite{BMPS2005} that sorted score vectors without duplicate values are also in bijection with sorted score sequences of round-robin tournaments with no draws (see OEIS sequence A000571).
\end{remark}

In this paper we resolve a weaker form of this conjecture. Recall that a \bfemph{doubly stochastic (DS) matrix} is a non-negative real matrix whose rows and columns all sum to $1$. In particular, the permutation matrices are the DS matrices with integer entries.

\begin{theorem}[Resolving \cite{Nordh2017}, Conjecture 15]\label{thm:main}
    If a score vector contains no duplicate values, then its frequency vector can always be realized as the X-ray of a doubly stochastic matrix.
\end{theorem}

The paper is organized as follows: In \cref{sec:lin_freq} and \cref{sec:lin_xray} we determine a set of linear constraints for a vector to be the frequency vector of a score vector, and one for a vector to be the X-ray of a doubly stochastic matrix. In \cref{sec:equiv} we show that these two systems of inequalities are equivalent when the vector is binary. In \cref{sec:sat_conj} we discuss a potential direction for fully resolving \cref{conj:binary} and some other related objects.

\section{Linear constraints on frequency vectors} \label{sec:lin_freq}

In order to work with X-rays for real-valued matrices, we need to first determine how the permutohedral constraints (in $\mathbb{R}^n$) translate to the frequency space (in $\mathbb{R}^{2n-1}$). 

\begin{proposition}\label{prop:x-conditions}
    An integer vector $s\in \{-n+1, \dots, n-1\}^n$ satisfies the permutohedral constraints (i.e.\ is a score vector) if and only if its corresponding frequency vector $x\in \mathbb{N}^{2n-1}$ satisfies the equation $S: \sum_{d=-n+1}^{n-1} x_d = n$ and the following inequalities for all integers $k = 0, \dots, n$: 
        \[D_k: \sum_{i=1}^{2k} (2k-i+1) x_{-n+i} \leq k(k+1) \quad \text{and} \quad U_k: \sum_{i=1}^{2k} (2k-i+1) x_{n-i}\leq k(k+1).\]
    (Here we let $x_{-n} = x_n = 0$.)
\end{proposition}

\begin{example}
    When $n = 4$, the constraints in the statement of \cref{prop:x-conditions} are 
    \begin{alignat*}{2}
        &S: & x_{-3} + x_{-2} + x_{-1} + x_0 + x_1 + x_2 + x_3 &= 4 \\
        &D_1: & 2x_{-3} + x_{-2} &\leq 2 \\
        &D_2: & 4x_{-3} + 3x_{-2} + 2x_{-1} + x_0 &\leq 6 \\
        &D_3: \quad & 6x_{-3} + 5x_{-2} + 4x_{-1} + 3x_0 + 2x_1 + x_2 &\leq 12 \\
        &U_1: & x_2 + 2x_3 &\leq 2 \\
        &U_2: & x_0 + 2x_1 + 3x_2 + 4x_3 &\leq 6 \\
        &U_3: & x_{-2} + 2x_{-1} + 3x_0 + 4x_1 + 5x_2 + 6x_3 &\leq 12.
    \end{alignat*}

    (Inequalities $D_0$ and $U_0$ are trivial ($0\leq 0$), while $D_4 = D_3 + 2S$ and $U_4 = U_3 + 2S$, so they are omitted from the list, although they will be used in the proof.)

    Intuitively, the inequalities $D_k$ and $U_k$ mean that the frequency vector should not be too concentrated towards either of the two ends. 
\end{example}

\begin{proof}[Proof of \cref{prop:x-conditions}]
    We first consider a score vector $s$ and a tournament with that score vector. If $s$ is a permutation of $(-n+1, -n+3, \dots, n-3, n-1)$, then the corresponding tournament is ``totally ordered'', where a higher-ranked player always wins over a lower-ranked player. In this case the corresponding frequency vector is $(1, 0, 1, \dots, 1, 0, 1)$, and it is not difficult to see that all inequalities in the statement are in fact equalities in this case. 
    
    Otherwise, it is not difficult to see that there exist two players $i$ and $j$ where $s_i \geq s_j$ but $i$ did not win over $j$. If we change the outcome of this match to be 1 point more in favor of $i$ (a draw into a win, or a loss into a draw), then the sum of squares of all scores strictly increases (since $(s_i+1)^2 + (s_j-1)^2 \geq s_i^2 + s_j^2 + 2$), so this process must terminate in finitely many steps, ending in a totally ordered tournament. By running this process in reverse, we see that any frequency vector of a score vector can be obtained via starting with $(1, 0, 1, \dots, 1, 0, 1)$ and repeatedly decrementing $x_i$ and $x_j$ by $1$ for some pair $(i,j)$ where $j - i \geq 2$ and $x_i, x_j\geq 1$ while incrementing $x_{i+1}$ and $x_{j-1}$ by $1$. It is clear that such operation never makes any of the listed inequalities false, so the final vector $x$ must also satisfy them.

    For the reverse direction, we first claim that the permutohedral constraints are equivalent to the following (quadratic) inequalities for all $t = 1, \dots, 2n-1$: \begin{align*}
        \sum_{i=1}^{t}(-n+i)x_{-n+i} &\geq -\left(\sum_{i=1}^{t}x_{-n+i}\right)\left(n-\sum_{i=1}^{t}x_{-n+i}\right), \\
        \sum_{i=1}^{t}(n-i)x_{n-i} &\leq \left(\sum_{i=1}^{t}x_{n-i}\right)\left(n-\sum_{i=1}^{t}x_{n-i}\right),
    \end{align*}
    or equivalently 
    \begin{equation}\label{eq2}
        \left(\sum_{i=1}^{t}x_{-n+i}\right)^2 \leq \sum_{i=1}^t ix_{-n+i}\quad \text{and}\quad \left(\sum_{i=1}^{t}x_{n-i}\right)^2 \leq \sum_{i=1}^t ix_{n-i}.
    \end{equation}
    These inequalities follow from the permutohedral constraints by letting $I$ be the set of all indices $i$ for which $s_i\leq -n+t$ (or $s_i\geq n-t$). 
    
    We can also show that these inequalities are sufficient. First, assume without loss of generality that the S-vector is sorted, so it suffices to prove the permutohedral constraints for $I = [k]$ (and $I = [n]\setminus [k]$) for each $k \in [n-1]$. Suppose that $I = [k]$ and $s_k = -n+t$, and that there are $x'\leq x_{-n+t}$ copies of $-n+t$ in the first $k$ coordinates. Since the first inequality in \eqref{eq2} is convex with respect to the variable $x_{-n+t}$, and it holds with $x_{-n+t}$ replaced by $0$ (which is the same as the inequality for $t-1$), it must also hold for with $x_{-n+t}$ replaced by $x'$ (since $0\leq x' \leq x_{-n+t}$). This is equivalent to the permutohedral constraint for $I = [k]$. The case for $I = [n]\setminus [k]$ is symmetric.

    Now, for each $t$ let $k = \sum_{i=1}^t x_{-n+i}$ (which is an integer in $[0, n]$). Since $(2k-i+1)x_{-n+i}$ is non-negative when $i\leq k$ and non-positive when $i > k$, by $D_k$ we have \[(2k+1)k-\sum_{i=1}^t ix_{-n+i} = \sum_{i=1}^{t} (2k-i+1) x_{-n+i} \leq \sum_{i=1}^{k} (2k-i+1) x_{-n+i}\leq k(k+1),\] which directly rearranges to the first inequality of \eqref{eq2}. The second inequality can be proved symmetrically.
\end{proof}

Obviously, the same conditions make sense when $x$ is a non-negative real-valued vector (even though such $x$ does not correspond to a score vector), so we obtain a polytope in $\mathbb{R}^{2n-1}$ whose integer points are frequency vectors of score vectors. We emphasize that this is not the same polytope as the one specified by the permutohedral constraints (which is in $\mathbb{R}^n$); the only connection that the authors are aware of is that it has the same number of integer points as the part of $2\Pi_n$ that lies in the region $s_1\leq s_2\leq \cdots \leq s_n$.

\section{Linear constraints on DS X-rays} \label{sec:lin_xray}

On the other end, there is a necessary and sufficient condition for when a vector is an X-ray of a doubly stochastic matrix using linear programming duality:

\begin{proposition}\label{prop:w-conditions}
    A non-negative vector $x\in \mathbb{R}^{2n-1}$ satisfying $\sum\limits_{d=-n+1}^{n-1} x_d = n$ is the X-ray of a doubly stochastic matrix if and only if for every non-negative real vector $w = (w_0, \dots, w_{n-1})$ the following inequality holds: \begin{equation}
        I(w):\quad \sum_{d=-n+1}^{n-1} v_d(w) x_d \leq 2\sum_{i=0}^{n-1} w_i,
    \end{equation} where $v_d(w) = \min_{\substack{0\leq i\leq j\leq n-1\\ i+j=d+(n-1)}} (w_i + w_j)$ for each $d$.
\end{proposition}

\begin{example}\label{ex:fund_weights}
    When $n\geq 3$ and $w = (2, 1, 0, \dots, 0)$, we see that $v(w) = (4, 3, 2, 1, 0, \dots, 0)$, so $I(w)$ is the inequality $4x_{-n+1}+3x_{-n+2}+2x_{-n+3}+x_{-n+4}\leq 6$, which is exactly inequality $D_2$. In fact, all inequalities $D_k$ and $U_k$ can be written in the form of $I(w)$ where $w = (k, k-1, \dots, 1, 0, \dots, 0)$ or $(0, \dots, 0, 1, \dots, k-1, k)$.
\end{example}

\begin{example}\label{ex:counter_example}
    When $n = 6$, the score vector $(-2, -2, -2, 2, 2, 2)$ corresponds to the frequency vector $x$ where $x_{-2} = x_2 = 3$ and all other $x_i$'s are zero. When $w = (0, 0, 1, 1, 0, 0)$, we see $v_d(w) = 1$ when $d = -2$ or $2$, and zero otherwise. Therefore $I(w)$ is the inequality $x_{-2} + x_{2} \leq 4$, which $x$ violates. This means there are no DS matrices with X-ray equal to $x$, let alone any permutations.

    Indeed, the sum of all entries in row 3, row 4, column 3, or column 4 of $A$ is always $4$ (counting entries in the intersection twice), but these entries also contain the two diagonals that must each sum to $3$, which is clearly impossible. See \cref{fig:visual} for a visual explanation.
\end{example}

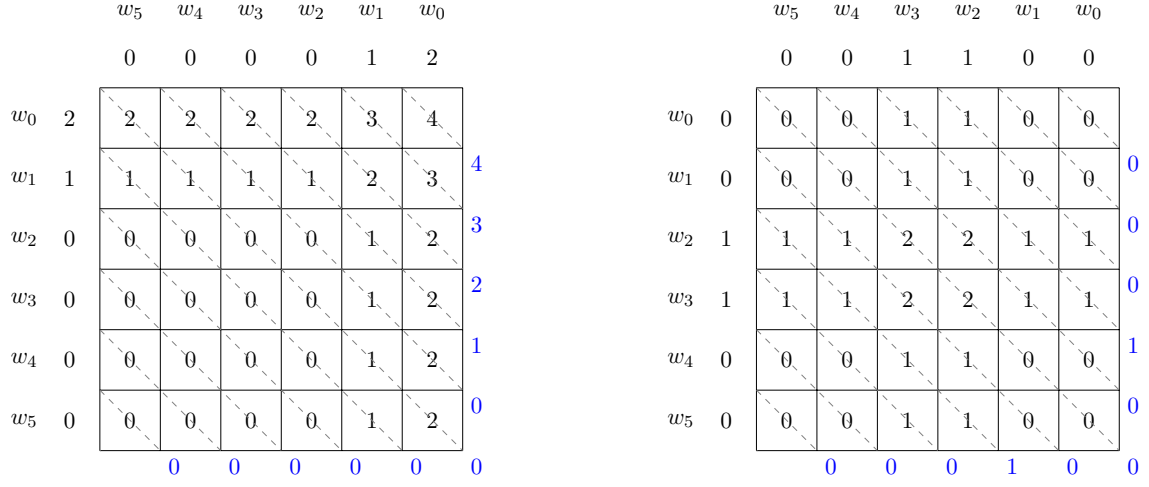
\begin{figure}[]
    \centering
    \begin{subfigure}[t]{0.45\textwidth}
        \centering
        \scalebox{0.8}{
        \begin{tikzpicture}
            \readlist\w{2,1,0,0,0,0};
            \readlist\v{4,3,2,1,0,0,0,0,0,0,0};
            \draw[step=1] (0,0) grid (6,6);
            \foreach \x in {0,1,2,3,4,5} {
                \node[] at (-1.25, 5.5-\x) {$w_{\x}$};
                \node[] at (-0.5, 5.5-\x) {\w[\x+1]};
                \node[] at (5.5-\x, 7.25) {$w_{\x}$};
                \node[] at (5.5-\x, 6.5) {\w[\x+1]};
            }
            \foreach \x in {1,2,3,4,5,6} {
                \foreach \y in {1,2,3,4,5,6} {
                    \node[] at (6.5-\x, 6.5-\y) {\pgfmathparse{\w[\x]+\w[\y]} \pgfmathprintnumber{\pgfmathresult}};
                }
            }
            \foreach \x in {0,1,2,3,4,5} {
                \draw[gray, dashed] (5-\x, 6) -- (6, 5-\x);
                \draw[gray, dashed] (0, 5-\x) -- (5-\x, 0);
            }
            \foreach \x in {1,2,3,4,5,6} {
                \node[anchor=north west] at (6,6-\x) {\textcolor{blue}{\v[\x]}};
            }
            \foreach \x in {1,2,3,4,5} {
                \node[anchor=north west] at (6-\x,0) {\textcolor{blue}{\v[6+\x]}};
            }
        \end{tikzpicture}
        }
    \end{subfigure}
    \hspace{1cm}
    \begin{subfigure}[t]{0.45\textwidth}
        \centering
        \scalebox{0.8}{
        \begin{tikzpicture}
            \readlist\w{0,0,1,1,0,0};
            \readlist\v{0,0,0,1,0,0,0,1,0,0,0};
            \draw[step=1] (0,0) grid (6,6);
            \foreach \x in {0,1,2,3,4,5} {
                \node[] at (-1.25, 5.5-\x) {$w_{\x}$};
                \node[] at (-0.5, 5.5-\x) {\w[\x+1]};
                \node[] at (5.5-\x, 7.25) {$w_{\x}$};
                \node[] at (5.5-\x, 6.5) {\w[\x+1]};
            }
            \foreach \x in {1,2,3,4,5,6} {
                \foreach \y in {1,2,3,4,5,6} {
                    \node[] at (6.5-\x, 6.5-\y) {\pgfmathparse{\w[\x]+\w[\y]} \pgfmathprintnumber{\pgfmathresult}};
                }
            }
            \foreach \x in {0,1,2,3,4,5} {
                \draw[gray, dashed] (5-\x, 6) -- (6, 5-\x);
                \draw[gray, dashed] (0, 5-\x) -- (5-\x, 0);
            }
            \foreach \x in {1,2,3,4,5,6} {
                \node[anchor=north west] at (6,6-\x) {\textcolor{blue}{\v[\x]}};
            }
            \foreach \x in {1,2,3,4,5} {
                \node[anchor=north west] at (6-\x,0) {\textcolor{blue}{\v[6+\x]}};
            }
        \end{tikzpicture}
        }
    \end{subfigure}
    \caption{Two examples of computing $v(w)$ (in blue) from $w$ by taking the minimum along each diagonal in an $n\times n$ matrix of $w_i+w_j$'s.}\label{fig:compute-v}
\end{figure}

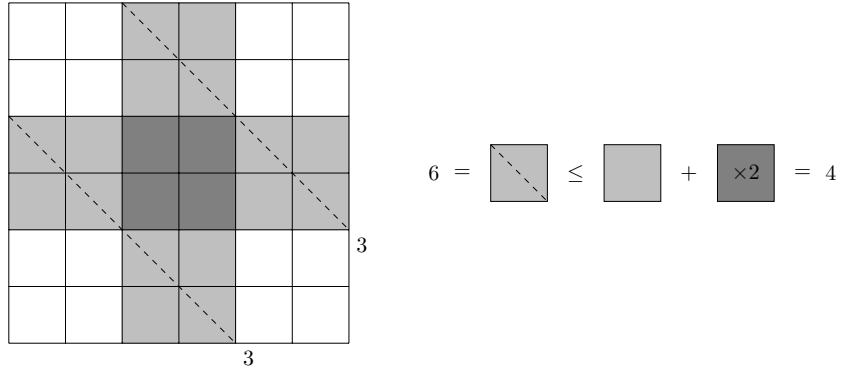
\begin{figure}[]
    \centering
    \scalebox{0.75}{
    \begin{tikzpicture} 
        \filldraw[lightgray] (0,2) rectangle (6,4);
        \filldraw[lightgray] (2,0) rectangle (4,6);
        \filldraw[gray] (2,2) rectangle (4,4);
        \draw[step=1] (0,0) grid (6,6);
        \draw[dashed] (2,6)--(6,2);
        \draw[dashed] (0,4)--(4,0);
        \node[anchor=north west] at (6,2) {3};
        \node[anchor=north west] at (4,0) {3};

        \node[] at (7.5,3) {6};
        \node[] at (8,3) {=};
        \filldraw[color=black, fill=lightgray] (8.5,2.5) rectangle (9.5,3.5);
        \draw[dashed] (8.5,3.5)--(9.5,2.5);
        \node[] at (10,3) {$\leq$};
        \filldraw[color=black, fill=lightgray] (10.5,2.5) rectangle (11.5,3.5);
        \node[] at (12,3) {$+$};
        \filldraw[color=black, fill=gray] (12.5,2.5) rectangle (13.5,3.5);
        \node[] at (13,3) {$\times 2$};
        \node[] at (14,3) {=};
        \node[] at (14.5,3) {4};
    \end{tikzpicture}
    }
    \caption{A proof that $(0,0,0,3,0,0,0,3,0,0,0)$ cannot be the X-ray of a DS matrix.}\label{fig:visual}
\end{figure}

\begin{proof}[Proof of \cref{prop:w-conditions}]
    Consider the set $\mathcal{P}\subseteq \mathbb{R}^{n\times n}$ of all $n\times n$ ``doubly sub-stochastic'' matrices (where row and column sums are \emph{at most} 1), which is a polytope defined by the following inequalities: 
    
    \begin{alignat*}{2}
        &E(i,j):& -a_{i,j} &\leq 0,\\
        &R(i):& \sum_{j'=1}^n a_{i,j'} &\leq 1,\\
        &C(j):& \sum_{i'=1}^n a_{i',j} &\leq 1.
    \end{alignat*}

    A (non-negative) vector $x\in \mathbb{R}^{2n-1}$ is the X-ray of a sub-DS matrix if and only if the affine subspace $\mathcal{H}_x\subseteq \mathbb{R}^{n\times n}$ given by the equalities $\sum_{\substack{i, j\in [n]\\ i-j=d}} a_{i,j} = x_d$ intersects $\mathcal{P}$. By the hyperplane separation theorem, $\mathcal{H}_x\cap \mathcal{P} = \varnothing$ if and only if there exists a non-zero vector $\bar{v}\in \mathbb{R}^{n\times n}$ such that $\min_{a\in \mathcal{H}_x} \bar{v}^Ta > \max_{a\in \mathcal{P}} \bar{v}^Ta$. Since $\mathcal{H}_x$ is an affine subspace, the value of $\bar{v}^Ta$ must be constant over all $a\in \mathcal{H}_x$, which in particular means that $\bar{v}_{i,j}$ takes the same value over each set of indices $(i,j)\in [n]^2$ for which $i-j = d$. If we write this value as $\bar{v}_d$ for each $d$, then $\min_{a\in \mathcal{H}_x} \bar{v}^Ta = \sum_{d=-n+1}^{n-1} \bar{v}_dx_d$, which we denote by $V$.

    On the other hand, by linear programming duality there exists a non-negative linear combination of the defining inequalities of $\mathcal{P}$ that amounts to some inequality $\sum_{i,j=1}^n \bar{v}_{i,j} a_{i,j} \leq W$ for some $W = \max_{a\in \mathcal{P}} \bar{v}^Ta < V$. If the coefficients of this linear combination are $e_{i,j}, r_i, c_j$ respectively, then we have $r_i + c_j - e_{i,j} = \bar{v}_{i,j} = \bar{v}_{i-j}$ for all $i, j\in [n]$, and \[W = \sum_{i=1}^n r_i + \sum_{j=1}^n c_j < \sum_{d=-n+1}^{n-1} \bar{v}_dx_d = V.\]
    
    Since all $x_d$'s are non-negative, we can increase each $\bar{v}_{d}$ without affecting this inequality as long as all $e_{i,j} = r_i + c_j - \bar{v}_{i-j} \geq 0$, so we can assume that $\bar{v}_{d} = \min_{\substack{i, j\in [n]\\ i-j=d}} (r_i+c_j)$ for each $d$. Moreover, if we let $w_i = \frac{1}{2}(r_{i+1} + c_{n-i})$ for $i = 0, 1, \dots, n-1$, then for each $d$ there is \[\min_{\substack{i, j\in [n]\\ i-j=d}} (r_i+c_j) \leq \min_{\substack{0\leq i, j\leq n-1\\ i+j=d+(n-1)}} (w_i + w_j)=v_d(w),\] and the equality can be achieved by taking $r_{i+1} = c_{n-i} (=w_i)$ for each $i$. Therefore, $\mathcal{H}_x\cap \mathcal{P} = \varnothing$ if and only if there exists $w_0, \dots, w_{n-1}\geq 0$ such that \[ 2\sum_{i=0}^{n-1} w_i < \sum_{d=-n+1}^{n-1} v_d(w)x_d,\] which is exactly the contrapositive of the statement, restricted to the hyperplane $\sum_d x_d = n$.
\end{proof}

\subsection{The facets of the X-ray polytope} Since the set of all X-rays of a DS matrix is simply a projection of the Birkhoff polytope, it will also be a polytope with finitely many facets, and hence we only need to check $I(w)$ for finitely many vectors $w$. While not necessary for the proof in the following section, we show here how to obtain this finite set for each $n$. 

First, observe that each X-ray $x$ satisfies two linear inequalities $\sum_{d} x_d = n$ and $\sum_{d} dx_d = 0$, which allows us to obtain \begin{equation}
    x_{-n+1} = \frac{n}{2}-\frac{1}{2(n-1)}\sum_{d=-n+2}^{n-2}(n-1-d)x_d, \quad x_{n-1} = \frac{n}{2} - \frac{1}{2(n-1)}\sum_{d=-n+2}^{n-2}(n-1+d)x_d. \label{eq:x_end}
\end{equation}

Plugging these two equations into $I(w)$ reduces it to \begin{equation}
    I'(u): \quad \sum_{d=-n+2}^{n-2} v_d(u) x_d \leq 2\sum_{i=1}^{n-2} u_i, \label{eq:i_prime}
\end{equation}under a linear change of variables $u_i = w_i - \frac{n-1-i}{n-1}w_0 - \frac{i}{n-1}w_{n-1}$ for $i = 0, \dots, n-1$. Note that $u_0 = u_{n-1} = 0$, and $(u_1, \dots, u_{n-2})$ ranges over all possible points in $\mathbb{R}^{n-2}$ as $w$ ranges over all non-negative real vectors (one can always find some $w$ that give rise to a given $u$ by letting $w_0$ and $w_{n-1}$ be sufficiently large).

Observe that all coefficients of $I'(u)$ are piecewise-linear functions of $u$, so each $I'(u)$ where $u$ is inside a fully-linear piece is implied by the ones on the boundary rays of that piece. Consequently, we obtain the following description of the facets:

\begin{proposition}\label{prop:rays}
    Let $\mathcal{T} = \{T_d\}_{d=-n+2}^{n-2}$ be a tropical hyperplane arrangement in $\mathbb{R}^{n-2}$, where each $T_d$ is the set of all points $u$ where the minimum in $v_d(u)=\min_{\substack{0\leq i\leq j\leq n-1\\ i+j=d+(n-1)}} (u_i + u_j)$ is attained at least twice (taking $u_0 = u_{n-1} = 0$).
    
    The projection map $\bar{x}$ that sends $A\in \mathbb{R}^{n\times n}$ to its ``trimmed X-ray'' $\bar{x}(A) = (x_{-n+2}, \dots, x_{n-2})$ sends the Birkhoff polytope to a $(2n-3)$-dimensional polytope. The facets of this polytope are either $x_d \geq 0$ for $d\in \{-n+1, \dots, n-1\}$ (where $x_{-n+1}$ and $x_{n-1}$ are given by \eqref{eq:x_end}) or have the form $I'(u)$ in \eqref{eq:i_prime}, where each $u$ is on one of the rays of $\mathcal{T}$.
\end{proposition}

Note that $I'(u)$ is invariant under scaling $u$ by a positive factor, so it suffices to compute all rays of $\mathcal{T}$ and take one point on each ray. It is unclear whether these rays have any simple descriptions.

\begin{example}
    Let us demonstrate this process in the case where $n = 5$. In this case, the inequality $I'(u)$ can be written as \begin{align*}
        u_1x_{-3} + \min(2u_1, u_2) x_{-2} + \min(u_1+u_2, u_3) x_{-1} + \min(0, u_1+u_3, 2u_2) x_0 & \\
        + \min(u_1, u_2+u_3)x_{1} + \min(u_2, 2u_3)x_{2} + u_3x_3 & \leq 2(u_1+u_2+u_3).
    \end{align*}
    The coefficients of $x_{-2}, x_{-1}, x_0, x_1, x_2$ each define a tropical hyperplane with $u_1, u_2, u_3$ as the variables, giving an arrangement as shown in \cref{fig:arr}. The direction vectors of the 16 rays of this arrangement are (up to scaling): \[\pm(1,2,3), \pm(3,2,1), \pm(1,2,1), \pm(1,2,-1), \pm(-1, 2,1), \pm(1,0,1), \pm(1,0,-1), (1,0,0), (0,0,1). \]
    Along with the 9 inequalities $x_d\geq 0$ for $-4\leq d \leq 4$, there are 25 potential facets for the projected polytope. We emphasize that they may not necessarily all correspond to facets: When $n = 5$ the X-ray polytope has only 23 facets, as the inequality $I'((1,0,0)): x_{-3} \leq 2$ is implied by summing the following two inequalities: \begin{alignat*}{2}
        &x_{-4}\geq 0:& 7x_{-3} + 6x_{-2} + 5x_{-1} + 4x_{0} + 3x_1 + 2x_2 + x_3 &\leq 20, \\
        &I'((-3, -2, -1)):\ & -3x_{-3}-6x_{-2}-5x_{-1}-4x_{0}-3x_1-2x_2-x_3&\leq -12,
    \end{alignat*}
    and similar for $I'((0,0,1))$.

    \begin{figure}[h!]
        \centering
        \includegraphics[scale=0.25]{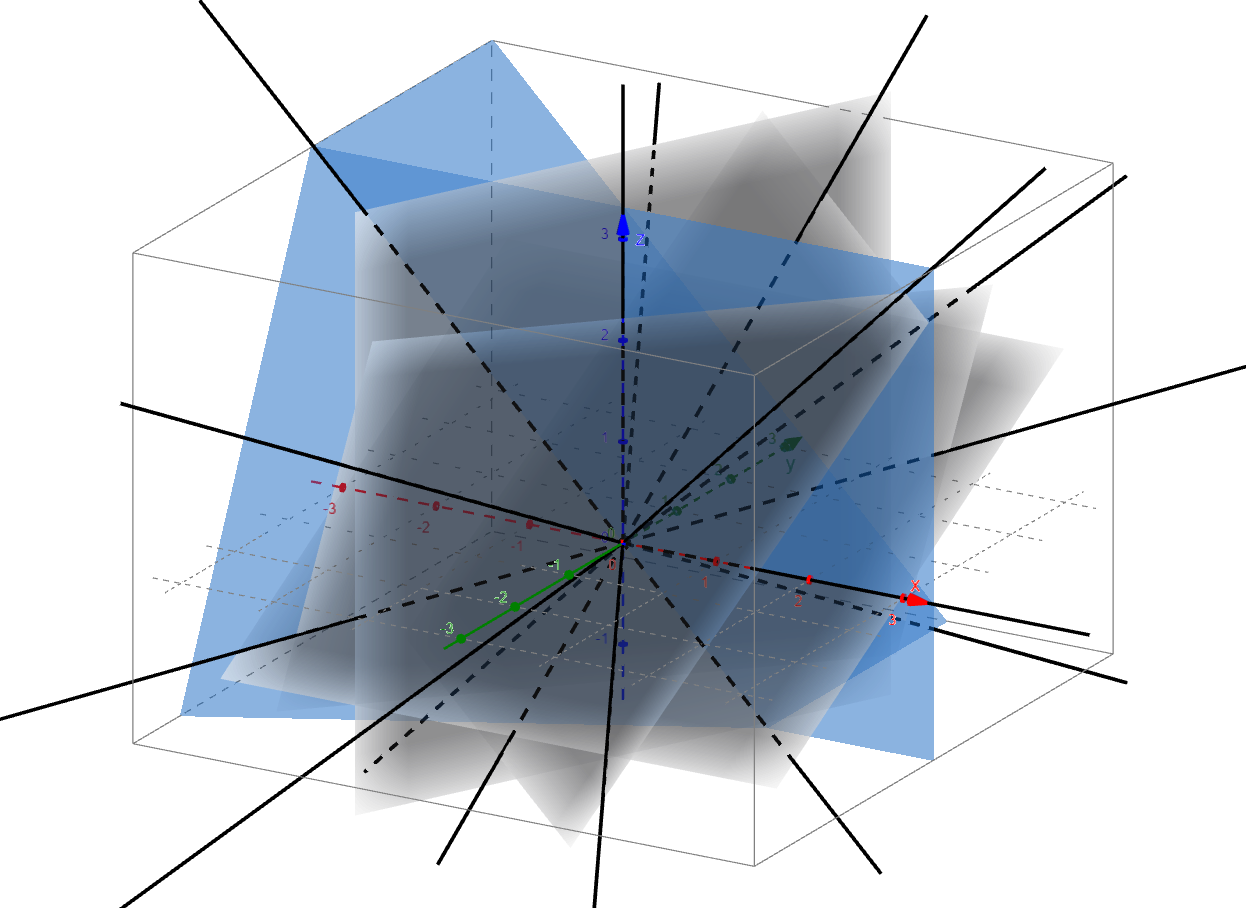}
        \caption{The arrangement $\mathcal{T}$ in \cref{prop:rays} and its rays for $n = 5$, drawn in Geogebra 3D. The tropical hyperplane $T_0$ is highlighted in blue to emphasize that they are not all normal hyperplanes.}\label{fig:arr}
    \end{figure}
\end{example}

\section{Equivalence of constraints in the binary case}\label{sec:equiv}

From \cref{ex:fund_weights} and \cref{ex:counter_example}, we see that $D_k$ and $U_k$ are special cases of $I(w)$ but do not imply all $I(w)$. However, if we restrict to the binary case, it turns out that they are in fact equivalent. In particular, we will show that every $I(w)$ can be written as a non-negative linear combination of $D_k, U_k$, and $x_d\leq 1$.

To this end, we first show a ``one-sided'' version of this equivalence.

\begin{proposition}\label{prop:one-sided}
    Let $x\in \mathbb{R}^{2n-1}$ be a vector such that $x_d\leq 1$ for every $d$. If $x$ satisfies $D_k$ for $1\leq k\leq n-1$, then it satisfies $I(w)$ for every non-negative vector $w\in\mathbb{R}^{n}$ where $w_{n-1} = 0$.
\end{proposition}

\begin{proof}
    We use strong induction on $n$, where the base case $n = 1$ is trivial. 
    
    For $n > 1$, let $m$ be an integer in $[n-1]$ that minimizes $w_{n-1-m}/m$, and let this minimum value be $\alpha$. Let $w' = w - \alpha\cdot (n-1, n-2, \dots, 1, 0)$ (so $w'_i = w_i - (n-1-i)\alpha$), then $w'$ is also a non-negative vector. Moreover, it is not difficult to see that $v_d(w') = v_d(w) - (n-1-d)\alpha$ for every $d$, so 
    \begin{align*}
        \sum_{d=-n+1}^{n-1} v_d(w)x_d &= \sum_{d=-n+1}^{n-1} v_d(w')x_d + \alpha \sum_{d=-n+1}^{n-1} (n-1-d)x_d \\
        &=\sum_{d=-n+1}^{n-1} v_d(w')x_d + \alpha\sum_{i=1}^{2n-2} (2n-1-i)x_{-n+i}.
    \end{align*}
    Since $2\sum_{i} w_i = 2\sum_{i} w'_i + \alpha(n-1)n$, we have $I(w) = I(w') + \alpha\cdot D_{n-1}$, and it suffices to show that $I(w')$ holds. 
    
    To do this, we split the LHS of $I(w')$ into three parts and bound separately: \begin{align*}
        \sum_{d=-n+1}^{n-1} v_d(w') x_d =&\ \sum_{d=-n+1}^{n-1-2m} v_d(w')x_d + \sum_{d=n-1-2m}^{n-1-m} v_d(w')x_d + \sum_{n-1-m}^{n-1} v_d(w')x_d \\
        \leq&\ \sum_{d=-n+1}^{n-1-2m} \left(\min_{\substack{0\leq i, j\leq n-1-m\\ i+j=d+(n-1)}} (w'_i+w'_j)\right) x_d + \sum_{d=n-1-2m}^{n-1-m} (w'_{n-1-m}+w'_{d+m})x_d \\ &\hspace{7cm}+\sum_{d=n-1-m}^{n-1} (w'_{n-1}+w'_{d})x_d \\
        \leq&\ 2\sum_{i=0}^{n-1-m} w'_i + \sum_{d=n-1-2m}^{n-1-m} w'_{d+m} + \sum_{d=n-1-m}^{n-1} w'_d\\
        =&\ 2\sum_{i=0}^{n-1-m} w'_i + 2\sum_{i=n-m}^{n-1} w'_i = 2\sum_{i=0}^{n-1} w'_i.
    \end{align*}
    In the second inequality, the first part is given by the induction hypothesis, and the second and third parts use the fact that $x_d\leq 1$ for all $d$. We also made use of the fact that $w'_{n-1-m} = w'_{n-1} = v_{n-1-2m}(w')= v_{n-1-m}(w')=0$ throughout this computation.
\end{proof}

By symmetry, the inequalities $U_k$ also imply all inequalities $I(w)$ where $w_0 =0$ if $0\leq x_d\leq 1$ for every $d$. We can put them together for a full ``two-sided'' equivalence.

\begin{proposition}\label{prop:two-sided}
    Let $x\in \mathbb{R}^{2n-1}$ be a vector such that $ x_d\leq 1$ for every $d\in [-n+1, n-1]$. If $x$ satisfies the $D_k$ and $U_k$ for $1\leq k\leq n-1$, then it satisfies $I(w)$ for every non-negative vector $w\in \mathbb{R}^n$.
\end{proposition}

\begin{proof}
    First, note that summing $D_{n-1}$ and $U_{n-1}$ gives \[\sum_{d=-n+1}^{n-1} (2n-2)x_d \leq 2(n-1)n \quad \Leftrightarrow\quad \sum_{d=-n+1}^{n-1} 2x_d \leq 2n \quad \Leftrightarrow\quad I(1,1,\dots,1).\]
    
    Let $w$ be a given non-negative real vector, and suppose that $w_m$ is (one of) the smallest entries of this vector for some index $m\in \{0, \dots, n-1\}$. It is not difficult to see that $I(w) = I(w - w_m\cdot \mathbf{1}) + w_m I(\mathbf{1})$, so $I(w-w_m\cdot \mathbf{1})$ implies $I(w)$. Therefore, we can assume without loss of generality that $w_m = 0$. In particular, we have $v_{2m-(n-1)}(w) = 0$.

    Now let $w^{\downarrow} = (w_0, w_1, \dots, w_{m-1}, w_m=0)$ (indexed from $0$ to $m$) and $w^{\uparrow} = (w_m = 0, w_{m+1}, \dots, w_{n-1})$ (indexed from $m$ to $n-1$). 
    Also let \[v_d^{\downarrow}(w) =v_d(w^{\downarrow})= \min_{\substack{0\leq i, j\leq m\\ i+j=d+(n-1)}} (w_i+w_j) \text{\quad and\quad} v_d^{\uparrow}(w) = v_d(w^{\uparrow}) = \min_{\substack{m\leq i, j\leq n-1\\ i+j=d+(n-1)}} (w_i+w_j)\] for $d\leq 2m-(n-1)$ and for $d\geq 2m-(n-1)$ respectively. It is not difficult to see that $v_d(w) \leq v_d^{\downarrow}(w)$ when $d\leq 2m-(n-1)$ and $v_d(w)\leq v_d^{\uparrow}(w)$ when $d\geq 2m-(n-1)$ as $v_d^{\downarrow}$ and $v_d^{\uparrow}$ restrict the set of $(i, j)$ that one takes the minimum over compared to $v_d$.
    
    We also note that $\sum_{i} w^{\downarrow}_i + \sum_{i} w^{\uparrow}_i = \sum_{i} w_i$, so if we show that the inequalities $I(w^{\downarrow})$ (when restricted to $d\leq 2m-(n-1)$) and $I(w^{\uparrow})$ (when restricted to $d\geq 2m-(n-1)$) are true, then summing the two will give an inequality that is (non-strictly) stronger than $I(w)$. Since these two inequalities are symmetric, it suffices to show that $I(w^{\downarrow})$ holds. This is exactly the content of \cref{prop:one-sided}.
\end{proof}

\begin{example}\label{proof-example}
    Let us demonstrate the proof of \cref{prop:one-sided} and \cref{prop:two-sided} more concretely with an example vector $w = (8, 5, 5, 1, 9, 2)$ (where $n = 6$). We recommend following along this example with reference to \cref{fig:proof-example}. 
    
    For this example, we have $v(w) = (16, 13, 10, 9, 6, 6, 2, 7, 3, 11, 4)$, so the inequality is 
    \begin{multline*}
        I(w):\ 16x_{-5}+13x_{-4}+10x_{-3}+9x_{-2}+6x_{-1}+6x_0+2x_{1}+7x_2+3x_3+11x_4+4x_5 \\ 
        \leq 2(8+5+5+1+9+2)=60.
    \end{multline*} 
    Subtracting the smallest entry $w_3 = 1$ from each $w_i$ amounts to subtracting $2$ from each $v_d(w)$ and hence $2\sum_{d}x_d$ from the left side and $2n = 12$ from the right side of $I(w)$, which makes the inequality stronger since $\sum_{d}x_d \leq n =6$. Therefore, it suffices to prove $I(w')$ where $w' = (7, 4, 4, 0, 8, 1)$ and $v(w') = (14, 11, 8, 7, 4, 4, 0, 5, 1, 9, 2)$, which has the form
    \begin{multline*}
        I(w'):\ 14x_{-5}+11x_{-4}+8x_{-3}+7x_{-2}+4x_{-1}+4x_0 + 0x_1 +5x_2+x_3+9x_4+2x_5 \\ 
        \leq 2(7+4+4+0+8+1)=48.
    \end{multline*} 

    Now, we let $w'^{\downarrow} = (7,4,4,0)$, $w'^{\uparrow} = (0, 8, 1)$, $v^{\downarrow}(w') = (14, 11, 8, 7, 4, 4, 0, {\color{gray}{0, 0, 0, 0}})$, and $v^{\uparrow}(w') = ({\color{gray}{0, 0, 0, 0, 0, 0,}}\, 0, 8, 1, 9, 2)$, which correspond to the inequalities \begin{alignat*}{2}
        &I(w'^{\downarrow}):\ & 14x_{-5}+11x_{-4}+8x_{-3}+7x_{-2}+4x_{-1}+4x_0+0x_1 &\leq 2(7+4+4+0)=30, \\
        &I(w'^{\uparrow}):& 0x_1+8x_2+x_3+9x_4+2x_5&\leq 2(0+8+1)=18.
    \end{alignat*}
    Because $v(w')\leq v^{\downarrow}(w') +v^{\uparrow}(w')$ coordinate-wise (notice that $v_2(w') = 5 < 8 = v_2^{\uparrow}(w')$ is strict because the minimum on the LHS is achieved at $i = 2$ (and $j=5$) which is outside the range $\{3,4,5\}$), so $I(w'^{\downarrow})$ and $I(w'^{\uparrow})$ together imply $I(w')$.

    We now show that $I(w'^{\downarrow})$ follows from the inequalities $D_1, D_2, D_3$ as well as $x_d\leq 1$. We compute $w'^{\downarrow}_0/3 = 7/3$, $w'^{\downarrow}_1/2 = 2$, $w'^{\downarrow}_2/1 = 4$, so $\alpha = \min(7/3, 2, 4) = 2$, and we subtract $2\cdot (3,2,1,0)$ from $w'^{\downarrow}$ to get $w''^{\downarrow} = (1,0,2,0)$. This also means that $v(w''^{\downarrow}) = v(w'^{\downarrow}) - 2(6,5,4,3,2, 1,0) = (2, 1, 0, 1, 0, 2, 0)$, which amounts to subtracting $12x_{-5}+10x_{-4}+\dots+2x_0$ from the left side and $4(3+2+1+0)=24$ from the right side of $I(w'^{\downarrow})$, giving the inequality \[I(w''^{\downarrow}):\ 2x_{-5}+x_{-4}+0x_{-3}+x_{-2}+0x_{-1}+2x_0 + 0x_1 \leq 6.\] This is the same as subtracting $2D_3$ from $I(w'^{\downarrow})$, so it suffices to prove $I(w''^{\downarrow})$. 
    Because $w''^{\downarrow}_1 = 0$, we can use the induction hypothesis on the first two coordinates of $w''^{\downarrow}$ to show that $2x_{-5}+x_{-4}+0x_{-3} \leq 2$ (in this case this is the same as $D_1$). For the remaining terms we note that the coefficients $v_{-3}, v_{-2}, v_{-1}, v_0, v_1$ are bounded above by \[(w''^{\downarrow}_1+w''^{\downarrow}_1, w''^{\downarrow}_1+w''^{\downarrow}_2, w''^{\downarrow}_1+w''^{\downarrow}_3, w''^{\downarrow}_2+w''^{\downarrow}_3, w''^{\downarrow}_3+w''^{\downarrow}_3) = (0, w''^{\downarrow}_2, 0, w''^{\downarrow}_2, 0)\] respectively and $x_d\leq 1$ for all $d$, so the sum $\sum_{d=-3}^{1} v_d(w''^{\downarrow})x_d$ is bounded by $2w''^{\downarrow}_2 = 4$. Combining the two gives $I(w''^{\downarrow})$, as desired. 
    
    The proof of $I(w'^{\uparrow})$ is similar, so working backwards we can prove $I(w')$ and $I(w)$.
\end{example}

\begin{figure}[h!]
    \centering
    \begin{subfigure}[t]{0.3\textwidth}
        \centering
        \scalebox{0.6}{
        \begin{tikzpicture}
            \readlist\w{8,5,5,1,9,2};
            \readlist\v{16,13,10,9,6,6,2,7,3,11,4};
            \draw[step=1] (0,0) grid (6,6);
            \foreach \x in {0,1,2,3,4,5} {
                \node[] at (-0.5, 5.5-\x) {\w[\x+1]};
                \node[] at (5.5-\x, 6.5) {\w[\x+1]};
            }
            \foreach \x in {1,2,3,4,5,6} {
                \foreach \y in {1,2,3,4,5,6} {
                    \node[] at (6.5-\x, 6.5-\y) {\pgfmathparse{\w[\x]+\w[\y]} \pgfmathprintnumber{\pgfmathresult}};
                }
            }
            \foreach \x in {0,1,2,3,4,5} {
                \draw[gray, dashed] (5-\x, 6) -- (6, 5-\x);
                \draw[gray, dashed] (0, 5-\x) -- (5-\x, 0);
            }
            \foreach \x in {1,2,3,4,5,6} {
                \node[anchor=north west] at (6,6-\x) {\textcolor{blue}{\v[\x]}};
            }
            \foreach \x in {1,2,3,4,5} {
                \node[anchor=north west] at (6-\x,0) {\textcolor{blue}{\v[6+\x]}};
            }
        \end{tikzpicture}
        }
    \end{subfigure}
    \raisebox{2cm}{$\ =\ $}
    \begin{subfigure}[t]{0.3\textwidth}
        \centering
        \scalebox{0.6}{
        \begin{tikzpicture}
            \readlist\w{1,1,1,1,1,1};
            \readlist\v{2,2,2,2,2,2,2,2,2,2,2};
            \draw[step=1] (0,0) grid (6,6);
            \foreach \x in {0,1,2,3,4,5} {
                \node[] at (-0.5, 5.5-\x) {\w[\x+1]};
                \node[] at (5.5-\x, 6.5) {\w[\x+1]};
            }
            \foreach \x in {1,2,3,4,5,6} {
                \foreach \y in {1,2,3,4,5,6} {
                    \node[] at (6.5-\x, 6.5-\y) {\pgfmathparse{\w[\x]+\w[\y]} \pgfmathprintnumber{\pgfmathresult}};
                }
            }
            \foreach \x in {0,1,2,3,4,5} {
                \draw[gray, dashed] (5-\x, 6) -- (6, 5-\x);
                \draw[gray, dashed] (0, 5-\x) -- (5-\x, 0);
            }
            \foreach \x in {1,2,3,4,5,6} {
                \node[anchor=north west] at (6,6-\x) {\textcolor{blue}{\v[\x]}};
            }
            \foreach \x in {1,2,3,4,5} {
                \node[anchor=north west] at (6-\x,0) {\textcolor{blue}{\v[6+\x]}};
            }
        \end{tikzpicture}
        }
    \end{subfigure}
    \raisebox{2cm}{$\ +\ $}
    \begin{subfigure}[t]{0.3\textwidth}
        \centering
        \scalebox{0.6}{
        \begin{tikzpicture}
            \readlist\w{7,4,4,0,8,1};
            \readlist\v{14,11,8,7,4,4,0,5,1,9,2};
            \draw[step=1] (0,0) grid (6,6);
            \foreach \x in {0,1,2,3,4,5} {
                \node[] at (-0.5, 5.5-\x) {\w[\x+1]};
                \node[] at (5.5-\x, 6.5) {\w[\x+1]};
            }
            \foreach \x in {1,2,3,4,5,6} {
                \foreach \y in {1,2,3,4,5,6} {
                    \node[] at (6.5-\x, 6.5-\y) {\pgfmathparse{\w[\x]+\w[\y]} \pgfmathprintnumber{\pgfmathresult}};
                }
            }
            \foreach \x in {0,1,2,3,4,5} {
                \draw[gray, dashed] (5-\x, 6) -- (6, 5-\x);
                \draw[gray, dashed] (0, 5-\x) -- (5-\x, 0);
            }
            \foreach \x in {1,2,3,4,5,6} {
                \node[anchor=north west] at (6,6-\x) {\textcolor{blue}{\v[\x]}};
            }
            \foreach \x in {1,2,3,4,5} {
                \node[anchor=north west] at (6-\x,0) {\textcolor{blue}{\v[6+\x]}};
            }
        \end{tikzpicture}
        }
    \end{subfigure} 
    \\ \vspace{0.9cm}
    \begin{subfigure}[t]{0.3\textwidth}
        \centering
        \scalebox{0.6}{
        \begin{tikzpicture}
            \readlist\w{7,4,4,0,8,1};
            \readlist\v{14,11,8,7,4,4,0,5,1,9,2};
            \draw[step=1] (0,0) grid (6,6);
            \foreach \x in {0,1,2,3,4,5} {
                \node[] at (-0.5, 5.5-\x) {\w[\x+1]};
                \node[] at (5.5-\x, 6.5) {\w[\x+1]};
            }
            \foreach \x in {1,2,3,4,5,6} {
                \foreach \y in {1,2,3,4,5,6} {
                    \node[] at (6.5-\x, 6.5-\y) {\pgfmathparse{\w[\x]+\w[\y]} \pgfmathprintnumber{\pgfmathresult}};
                }
            }
            \foreach \x in {0,1,2,3,4,5} {
                \draw[gray, dashed] (5-\x, 6) -- (6, 5-\x);
                \draw[gray, dashed] (0, 5-\x) -- (5-\x, 0);
            }
            \foreach \x in {1,2,3,4,5,6} {
                \node[anchor=north west] at (6,6-\x) {\textcolor{blue}{\v[\x]}};
            }
            \foreach \x in {1,2,3,4,5} {
                \node[anchor=north west] at (6-\x,0) {\textcolor{blue}{\v[6+\x]}};
            }
        \end{tikzpicture}
        }
    \end{subfigure} 
    \raisebox{2cm}{$\ \leq\ $}
    \begin{subfigure}[t]{0.3\textwidth}
        \centering
        \scalebox{0.6}{
        \begin{tikzpicture}
            \readlist\w{7,4,4,0,0,0};
            \readlist\v{14,11,8,7,4,4,0,0,0,0,0};
            \draw[step=1] (0,0) grid (6,6);
            \foreach \x in {0,1,2,3} {
                \node[] at (-0.5, 5.5-\x) {\w[\x+1]};
                \node[] at (5.5-\x, 6.5) {\w[\x+1]};
            }
            \foreach \x in {1,2,3,4} {
                \foreach \y in {1,2,3,4} {
                    \node[] at (6.5-\x, 6.5-\y) {\pgfmathparse{\w[\x]+\w[\y]} \pgfmathprintnumber{\pgfmathresult}};
                }
            }
            \foreach \x in {0,1,2,3,4,5} {
                \draw[gray, dashed] (5-\x, 6) -- (6, 5-\x);
                \draw[gray, dashed] (0, 5-\x) -- (5-\x, 0);
            }
            \foreach \x in {1,2,3,4,5,6} {
                \node[anchor=north west] at (6,6-\x) {\textcolor{blue}{\v[\x]}};
            }
            \foreach \x in {1,2,3,4,5} {
                \node[anchor=north west] at (6-\x,0) {\textcolor{blue}{\v[6+\x]}};
            }
        \end{tikzpicture}
        }
    \end{subfigure}
    \raisebox{2cm}{$\ +\ $}
    \begin{subfigure}[t]{0.3\textwidth}
        \centering
        \scalebox{0.6}{
        \begin{tikzpicture}
            \readlist\w{0,0,0,0,8,1};
            \readlist\v{0,0,0,0,0,0,0,8,1,9,2};
            \draw[step=1] (0,0) grid (6,6);
            \foreach \x in {3,4,5} {
                \node[] at (-0.5, 5.5-\x) {\w[\x+1]};
                \node[] at (5.5-\x, 6.5) {\w[\x+1]};
            }
            \foreach \x in {4,5,6} {
                \foreach \y in {4,5,6} {
                    \node[] at (6.5-\x, 6.5-\y) {\pgfmathparse{\w[\x]+\w[\y]} \pgfmathprintnumber{\pgfmathresult}};
                }
            }
            \foreach \x in {0,1,2,3,4,5} {
                \draw[gray, dashed] (5-\x, 6) -- (6, 5-\x);
                \draw[gray, dashed] (0, 5-\x) -- (5-\x, 0);
            }
            \foreach \x in {1,2,3,4,5,6} {
                \node[anchor=north west] at (6,6-\x) {\textcolor{blue}{\v[\x]}};
            }
            \foreach \x in {1,2,3,4,5} {
                \node[anchor=north west] at (6-\x,0) {\textcolor{blue}{\v[6+\x]}};
            }
        \end{tikzpicture}
        }
    \end{subfigure} 
    \\ \vspace{0.9cm}
    \begin{subfigure}[t]{0.3\textwidth}
        \centering
        \scalebox{0.6}{
        \begin{tikzpicture}
            \readlist\w{7,4,4,0};
            \readlist\v{14,11,8,7,4,4,0};
            \draw[step=1] (0,0) grid (4,4);
            \foreach \x in {0,1,2,3} {
                \node[] at (-0.5, 3.5-\x) {\w[\x+1]};
                \node[] at (3.5-\x, 4.5) {\w[\x+1]};
            }
            \foreach \x in {1,2,3,4} {
                \foreach \y in {1,2,3,4} {
                    \node[] at (4.5-\x, 4.5-\y) {\pgfmathparse{\w[\x]+\w[\y]} \pgfmathprintnumber{\pgfmathresult}};
                }
            }
            \foreach \x in {0,1,2,3} {
                \draw[gray, dashed] (3-\x, 4) -- (4, 3-\x);
                \draw[gray, dashed] (0, 3-\x) -- (3-\x, 0);
            }
            \foreach \x in {1,2,3,4} {
                \node[anchor=north west] at (4,4-\x) {\textcolor{blue}{\v[\x]}};
            }
            \foreach \x in {1,2,3} {
                \node[anchor=north west] at (4-\x,0) {\textcolor{blue}{\v[4+\x]}};
            }
        \end{tikzpicture}
        }
    \end{subfigure}
    \raisebox{1.5cm}{$\ =\ $}
    \begin{subfigure}[t]{0.3\textwidth}
        \centering
        \scalebox{0.6}{
        \begin{tikzpicture}
            \readlist\w{6,4,2,0};
            \readlist\v{12,10,8,6,4,2,0};
            \draw[step=1] (0,0) grid (4,4);
            \foreach \x in {0,1,2,3} {
                \node[] at (-0.5, 3.5-\x) {\w[\x+1]};
                \node[] at (3.5-\x, 4.5) {\w[\x+1]};
            }
            \foreach \x in {1,2,3,4} {
                \foreach \y in {1,2,3,4} {
                    \node[] at (4.5-\x, 4.5-\y) {\pgfmathparse{\w[\x]+\w[\y]} \pgfmathprintnumber{\pgfmathresult}};
                }
            }
            \foreach \x in {0,1,2,3} {
                \draw[gray, dashed] (3-\x, 4) -- (4, 3-\x);
                \draw[gray, dashed] (0, 3-\x) -- (3-\x, 0);
            }
            \foreach \x in {1,2,3,4} {
                \node[anchor=north west] at (4,4-\x) {\textcolor{blue}{\v[\x]}};
            }
            \foreach \x in {1,2,3} {
                \node[anchor=north west] at (4-\x,0) {\textcolor{blue}{\v[4+\x]}};
            }
        \end{tikzpicture}
        }
    \end{subfigure}
    \raisebox{1.5cm}{$\ +\ $}
    \begin{subfigure}[t]{0.3\textwidth}
        \centering
        \scalebox{0.6}{
        \begin{tikzpicture}
            \readlist\w{1,0,2,0};
            \readlist\v{2,1,0,1,0,2,0};
            \draw[step=1] (0,0) grid (4,4);
            \foreach \x in {0,1,2,3} {
                \node[] at (-0.5, 3.5-\x) {\w[\x+1]};
                \node[] at (3.5-\x, 4.5) {\w[\x+1]};
            }
            \foreach \x in {1,2,3,4} {
                \foreach \y in {1,2,3,4} {
                    \node[] at (4.5-\x, 4.5-\y) {\pgfmathparse{\w[\x]+\w[\y]} \pgfmathprintnumber{\pgfmathresult}};
                }
            }
            \foreach \x in {0,1,2,3} {
                \draw[gray, dashed] (3-\x, 4) -- (4, 3-\x);
                \draw[gray, dashed] (0, 3-\x) -- (3-\x, 0);
            }
            \foreach \x in {1,2,3,4} {
                \node[anchor=north west] at (4,4-\x) {\textcolor{blue}{\v[\x]}};
            }
            \foreach \x in {1,2,3} {
                \node[anchor=north west] at (4-\x,0) {\textcolor{blue}{\v[4+\x]}};
            }
        \end{tikzpicture}
        }
    \end{subfigure}
    \\ \vspace{0.9cm}
    \begin{subfigure}[t]{0.3\textwidth}
        \centering
        \scalebox{0.6}{
        \begin{tikzpicture}
            \readlist\w{1,0,2,0};
            \readlist\v{2,1,0,1,0,2,0};
            \draw[step=1] (0,0) grid (4,4);
            \foreach \x in {0,1,2,3} {
                \node[] at (-0.5, 3.5-\x) {\w[\x+1]};
                \node[] at (3.5-\x, 4.5) {\w[\x+1]};
            }
            \foreach \x in {1,2,3,4} {
                \foreach \y in {1,2,3,4} {
                    \node[] at (4.5-\x, 4.5-\y) {\pgfmathparse{\w[\x]+\w[\y]} \pgfmathprintnumber{\pgfmathresult}};
                }
            }
            \foreach \x in {0,1,2,3} {
                \draw[gray, dashed] (3-\x, 4) -- (4, 3-\x);
                \draw[gray, dashed] (0, 3-\x) -- (3-\x, 0);
            }
            \foreach \x in {1,2,3,4} {
                \node[anchor=north west] at (4,4-\x) {\textcolor{blue}{\v[\x]}};
            }
            \foreach \x in {1,2,3} {
                \node[anchor=north west] at (4-\x,0) {\textcolor{blue}{\v[4+\x]}};
            }
        \end{tikzpicture}
        }
    \end{subfigure}
    \raisebox{1.5cm}{$\ \leq\ $}
    \begin{subfigure}[t]{0.3\textwidth}
        \centering
        \scalebox{0.6}{
        \begin{tikzpicture}
            \readlist\w{1,0,0,0};
            \readlist\v{2,1,0,0,0,0,0};
            \draw[step=1] (0,0) grid (4,4);
            \foreach \x in {0,1} {
                \node[] at (-0.5, 3.5-\x) {\w[\x+1]};
                \node[] at (3.5-\x, 4.5) {\w[\x+1]};
            }
            \foreach \x in {1,2} {
                \foreach \y in {1,2} {
                    \node[] at (4.5-\x, 4.5-\y) {\pgfmathparse{\w[\x]+\w[\y]} \pgfmathprintnumber{\pgfmathresult}};
                }
            }
            \foreach \x in {0,1,2,3} {
                \draw[gray, dashed] (3-\x, 4) -- (4, 3-\x);
                \draw[gray, dashed] (0, 3-\x) -- (3-\x, 0);
            }
            \foreach \x in {1,2,3,4} {
                \node[anchor=north west] at (4,4-\x) {\textcolor{blue}{\v[\x]}};
            }
            \foreach \x in {1,2,3} {
                \node[anchor=north west] at (4-\x,0) {\textcolor{blue}{\v[4+\x]}};
            }
        \end{tikzpicture}
        }
    \end{subfigure}
    \raisebox{1.5cm}{$\ +\ $}
    \begin{subfigure}[t]{0.3\textwidth}
        \centering
        \scalebox{0.6}{
        \begin{tikzpicture}
            \readlist\w{0,0,2,0};
            \readlist\v{0,0,0,2,0,2,0};
            \draw[step=1] (0,0) grid (4,4);
            \foreach \x in {1,2,3} {
                \node[] at (-0.5, 3.5-\x) {\w[\x+1]};
                \node[] at (3.5-\x, 4.5) {\w[\x+1]};
            }
            \foreach \x in {2,3,4} {
                \node[] at (4.5-\x, 0.5) {\pgfmathparse{\w[\x]} \pgfmathprintnumber{\pgfmathresult}};
            }
            \foreach \y in {2,3} {
                \node[] at (2.5, 4.5-\y) {\pgfmathparse{\w[\y]} \pgfmathprintnumber{\pgfmathresult}};
            }
            \foreach \x in {0,1,2,3} {
                \draw[gray, dashed] (3-\x, 4) -- (4, 3-\x);
                \draw[gray, dashed] (0, 3-\x) -- (3-\x, 0);
            }
            \foreach \x in {1,2,3,4} {
                \node[anchor=north west] at (4,4-\x) {\textcolor{blue}{\v[\x]}};
            }
            \foreach \x in {1,2,3} {
                \node[anchor=north west] at (4-\x,0) {\textcolor{blue}{\v[4+\x]}};
            }
        \end{tikzpicture}
        }
    \end{subfigure}
    \caption{Four key reductions used in \cref{proof-example}, following the same notation as \cref{fig:compute-v} to represent each $I(w)$.}\label{fig:proof-example}
\end{figure}

Combining \cref{prop:w-conditions} and \cref{prop:two-sided} gives the following:

\begin{corollary}
    Let $x\in \mathbb{R}^{2n-1}$ be a vector such that $\sum_{d=-n+1}^{n-1} x_d = n$ and $0\leq x_d\leq 1$ for every $d$. It is the X-ray of a doubly stochastic matrix if and only if it satisfies inequality $D_k$ and $U_k$ for every $1\leq k\leq n-1$.
\end{corollary}

\cref{thm:main} is simply a special case of this corollary (where $x$ is an integer vector) due to \cref{prop:x-conditions}.

\begin{remark}\label{rem:sub}
    If we drop the condition that $\sum_{d} x_d = n$ from the statement, our proofs actually show that $D_k$ and $U_k$ are necessary and sufficient for a real-valued ``sub-binary'' vector $x$ to be the X-ray of a doubly sub-stochastic matrix.
\end{remark}

\section{A false saturation-type conjecture} \label{sec:sat_conj}

As noted in \cite{Nordh2017}, \cref{conj:binary} can now be reduced to the following conjecture:

\begin{conjecture}[\cite{Nordh2017}, Conjecture 17]\label{conj:ds-to-perm}
    If a binary vector $x\in \mathbb{R}^{2n-1}$ can be realized as the X-ray of a doubly stochastic matrix, then it can be realized as the X-ray of a permutation matrix.
\end{conjecture}

This conjecture admits a natural generalization:

\begin{definition}
    A \bfemph{triangular contingency table} of size $n$ is an array $A$ of non-negative real numbers $a_{i,j,k}$, where $(i, j, k)$ ranges over triples of non-negative integers with sum $n-1$, along with $3n$ specified \bfemph{row sums} or \bfemph{marginals} $\lambda_i, \mu_j, \nu_k$ for $0\leq i, j, k\leq n-1$, where \[\lambda_i = \sum_{j'+k'=n-1-i} a_{i,j',k'}, \quad \mu_j = \sum_{i'+k'=n-1-j} a_{i',j,k'}, \quad \nu_k = \sum_{i'+j'=n-1-k} a_{i',j',k}.\]
\end{definition}

It is not difficult to see that an $n\times n$ doubly stochastic matrix with specified X-ray $x$ is equivalent to a triangular contingency table of size $2n-1$ where $\lambda_i = x_{i-(n-1)}$ for each $i$, and $\mu_j, \nu_k = 1$ if $j, k\leq n-1$ and $0$ otherwise.

\begin{conjecture}\label{conj:contingency}
    If there is a triangular contingency table with specified integer marginals $\lambda_i, \mu_j, \nu_k$, then there exists such a table consisting of integer entries with the same marginals.
\end{conjecture}

Note that for normal (rectangular) contingency tables, this statement is easy as the matrix defining the row/column sum constraints is totally unimodular. (The set of all such tables with given marginals is called a \bfemph{transportation polytope}, which has been extensively studied; see \cite{LK2013} for an overview.) 
It is in fact possible to turn any contingency table into an integer one by only modifying the non-integer entries. This is not the case even for DS matrices (with a given binary X-ray), as shown by the matrix \[\begin{bmatrix} 1/2 & 1/2 & 0 & 0 & 0 \\ 1/2 & 0 & 0 & 1/2 & 0 \\ 0 & 0 & 1/2 & 0 & 1/2 \\ 0 & 1/2 & 0 & 0 & 1/2 \\ 0 & 0 & 1/2 & 1/2 & 0\end{bmatrix},\] where it is impossible to replace some $1/2$'s with $1$'s and the others with $0$'s while preserving the X-ray $(0,0,1,1,1,1,1,0,0)$. Do note that the the permutation matrix \[\begin{bmatrix} 0 & 1 & 0 & 0 & 0 \\ 1 & 0 & 0 & 0 & 0 \\ 0 & 0 & 0 & 0 & 1 \\ 0 & 0 & 0 & 1 & 0 \\ 0 & 0 & 1 & 0 & 0\end{bmatrix}\] has the same X-ray, so this does not constitute a counter-example to \cref{conj:ds-to-perm}.

Triangular contingency tables are also very similar to \bfemph{Berenstein--Zelevinsky (BZ) patterns}, where the marginal conditions are essentially the same and entry-wise non-negativity is replaced by non-negativity of prefix sums along each row (see \cite{Stanley2001}, Definition A.1.3.10). Notably, the number of integer BZ patterns with given marginals corresponds to Littlewood-Richardson coefficients (\cite{BZ1992}). Due to the celebrated proof of the saturation conjecture by Knutson and Tao (\cite{KT1999}), any real-valued BZ pattern with integer marginals can be turned into an integer pattern. \cref{conj:contingency} can hence be seen as a version of the saturation conjecture restricted to the subset of ``totally non-negative'' BZ patterns. We also note that this subset trivially respects $S_3$ symmetry of the marginals, whereas BZ patterns only have $C_3$ symmetry.

Unfortunately, the problem of determining if an integer triangular contingency table with given marginals exists is NP-complete, even when the marginals are restricted to be binary (\cite{BDGV2008}, Theorem 1). This means that one generally does not expect an easy criterion similar to the inequalities $L_k$ and $U_k$. Moreover, following the same reduction techniques in the proof of NP-completeness provides an explicit counter-example to \cref{conj:contingency}.

\begin{proposition}
    There exist binary marginals $\lambda_i, \mu_j, \nu_k \in \{0, 1\}$ for which there is a unique triangular contingency table with these marginals, and this table is not integer-valued.
\end{proposition}

\begin{proof}
    We begin with a \bfemph{3-way contigency table} $B$, which is a 3-dimensional array of non-negative real numbers $b_{i,j,k}$ where $i\in \{0, 1, 2\}, j\in \{0, 1, 2, 3\}, k\in \{0, 1, 2, 3, 4, 5\}$, with (2-way) marginals \[\alpha_{i,j} = \sum_{k'} b_{i,j,k'}, \quad \beta_{i,k} = \sum_{j'} b_{i,j',k}, \quad \gamma_{j,k} = \sum_{i'} b_{i',j,k},\] specified by the following tables:

    \begin{table}[h!]
    \begin{tabular}{cc|cccc}
     & $j$ & 0 & 1 & 2 & 3 \\
    $i$ & $\alpha_{i,j}$ &  &  &  &  \\ \hline
    0 &  & 1 & 1 & 1 & 1 \\
    1 &  & 1 & 1 & 1 & 1 \\
    2 &  & 1 & 1 & 1 & 1
    \end{tabular} \qquad
    \begin{tabular}{cc|cccccc}
     & $k$ & 0 & 1 & 2 & 3 & 4 & 5 \\
    $i$ & $\beta_{i,k}$ &  &  &  &  &  &  \\ \hline
    0 &  & 1 & 1 & 1 & 1 & 0 & 0 \\
    1 &  & 1 & 1 & 0 & 0 & 1 & 1 \\
    2 &  & 0 & 0 & 1 & 1 & 1 & 1
    \end{tabular} \qquad
    \begin{tabular}{cc|cccccc}
     & $k$ & 0 & 1 & 2 & 3 & 4 & 5 \\
    $j$ & $\gamma_{j,k}$ &  &  &  &  &  &  \\ \hline
    0 &  & 1 & 0 & 1 & 0 & 1 & 0 \\
    1 &  & 1 & 0 & 0 & 1 & 0 & 1 \\
    2 &  & 0 & 1 & 1 & 0 & 0 & 1 \\
    3 &  & 0 & 1 & 0 & 1 & 1 & 0
    \end{tabular}
    \end{table}

    From \cite{LK2013}, Example 3.4, there exists a unique $B$ with these marginals, where $b_{i, j, k} = 1/2$ if $\alpha_{i,j}=\beta_{i,k}=\gamma_{j,k}=1$, and $0$ otherwise.

    We now construct the following marginals for a triangular contingency table $A$ of size $n = 112$: \[\lambda_{4i+j}=\alpha_{i,j},\quad \mu_{8+20k-4i}=\beta_{i,k},\quad \nu_{103-20k-j}=\gamma_{j,k},\] and all other marginals not defined above are zero. It is not difficult to see that any entry in $A$ that is not of the form $a_{4i+j, 8+20k-4i, 103-20k-j}$ must be zero because it is on a row with sum zero (in some direction), and the entries that are of this form are in bijection with the entries $b_{i,j,k}$ in $B$, and must take the same value as $b_{i,j,k}$ to satisfy the marginals. (In essence, we projected the three-dimensional array $B$ onto a two-dimensional plane so that the rows of $B$ do not create extraneous three-way intersections after projection.) Consequently, the table $A$ is also unique and not integer-valued, as desired.
\end{proof}

Note that in the marginals we constructed in the above proof, there is $\lambda_{0} = \lambda_1 = \dots = \lambda_{11} = 1$; i.e.\ $\lambda$ has contiguous support. A counter-example where \emph{two} of the three sets of marginals have contiguous support is equivalent to a counter-example to \cref{conj:ds-to-perm}. This also means that any potential proof of \cref{conj:ds-to-perm} must strongly make use of the fact that the row and column sums are exactly $1$, unlike the proof of \cref{thm:main} (see \cref{rem:sub}). 

Similarly, the NP-completeness of the original \cref{question} suggests that the word ``binary'' from \cref{conj:ds-to-perm} cannot be dropped either. An explicit counter-example to the strengthened conjecture can be found via computer search with the help of \cref{prop:rays}: When $n = 12$, the DS matrix
\[
\begin{bmatrix}
    0 & 0 & 0 & 1/2 & 0 & 0 & 0 & 1/2 & 0 & 0 & 0 & 0 \\
    1/4 & 0 & 0 & 0 & 1/4 & 0 & 0 & 0 & 0 & 1/2 & 0 & 0 \\
    0 & 0 & 0 & 0 & 0 & 1/2 & 0 & 0 & 0 & 0 & 1/2 & 0 \\
    0 & 0 & 0 & 0 & 0 & 0 & 1 & 0 & 0 & 0 & 0 & 0 \\
    0 & 0 & 0 & 0 & 0 & 0 & 0 & 1/2 & 0 & 0 & 0 & 1/2 \\
    0 & 0 & 0 & 0 & 0 & 0 & 0 & 0 & 1 & 0 & 0 & 0 \\
    0 & 0 & 0 & 0 & 0 & 1/2 & 0 & 0 & 0 & 1/2 & 0 & 0 \\
    3/4 & 0 & 0 & 0 & 0 & 0 & 0 & 0 & 0 & 0 & 1/4 & 0 \\
    0 & 1/2 & 0 & 0 & 0 & 0 & 0 & 0 & 0 & 0 & 0 & 1/2 \\
    0 & 1/2 & 1/2 & 0 & 0 & 0 & 0 & 0 & 0 & 0 & 0 & 0 \\
    0 & 0 & 1/2 & 1/2 & 0 & 0 & 0 & 0 & 0 & 0 & 0 & 0 \\
    0 & 0 & 0 & 0 & 3/4 & 0 & 0 & 0 & 0 & 0 & 1/4 & 0 
\end{bmatrix}
\]
has X-ray $(0, 0, 0, 1, 1, 0, 0, 0, 5, 0, 0, 0, 1, 0, 0, 0, 0, 0, 3, 1, 0, 0, 0)$, which cannot be realized by any permutation matrix.

Nevertheless, it may still be worth attempting to imitate the proof of the saturation conjecture to obtain partial results in this direction, which may be helpful towards resolving the weaker (and still open) \cref{conj:ds-to-perm}. It may also be of independent interest to study these totally non-negative BZ patterns and find more combinatorial interpretations for the numbers of such patterns, similar to Littlewood-Richardson coefficients.

\bigskip

\paragraph{\bfseries Acknowledgments.} The authors would like to thank Alexander Postnikov and Ilani Axelrod-Freed for introducing X-rays of permutations to them. The authors would also like to thank the MIT Math Department for sponsoring this project as part of UROP(+) in summer and fall of 2025.

\bibliographystyle{alpha}
\bibliography{refs}

@article{BF2014,
title = {Hankel and Toeplitz X-rays of permutations},
journal = {Linear Algebra and its Applications},
volume = {449},
pages = {350-380},
year = {2014},
doi = {https://doi.org/10.1016/j.laa.2014.02.037},
url = {https://www.sciencedirect.com/science/article/pii/S0024379514001049},
author = {Richard A. Brualdi and Eliseu Fritscher}
}

@article{BMPS2005,
title = {On the X-rays of permutations},
journal = {Electronic Notes in Discrete Mathematics},
volume = {20},
pages = {193-203},
year = {2005},
note = {Proceedings of the Workshop on Discrete Tomography and its Applications},
doi = {https://doi.org/10.1016/j.endm.2005.05.063},
url = {https://www.sciencedirect.com/science/article/pii/S1571065305050675},
author = {Cecilia Bebeacua and Toufik Mansour and Alex Postnikov and Simone Severini}
}

@article{BDGV2008,
title = {On the reconstruction of binary and permutation matrices under (binary) tomographic constraints},
journal = {Theoretical Computer Science},
volume = {406},
number = {1},
pages = {63-71},
year = {2008},
note = {Discrete Tomography and Digital Geometry: In memory of Attila Kuba},
doi = {https://doi.org/10.1016/j.tcs.2008.06.014},
url = {https://www.sciencedirect.com/science/article/pii/S0304397508004131},
author = {S. Brunetti and A. {Del Lungo} and P. Gritzmann and S. {de Vries}},
}

@misc{Nordh2017,
      title={A note on X-rays of permutations and a problem of Brualdi and Fritscher}, 
      author={Gustav Nordh},
      year={2017},
      eprint={1707.03928},
      archivePrefix={arXiv},
      primaryClass={math.CO},
      url={https://arxiv.org/abs/1707.03928}, 
}

@article{LK2013,
    title = {Combinatorics and Geometry of Transportation Polytopes: An Update},
    author={Jes\'{u}s A. De Loera and Edward D. Kim},
    journal = {Comtemporary Math.},
    year = {2013},
    month = {06},
    volume = {625},
    doi = {10.1090/conm/625/12491}
}

@Article{BZ1992,
author={Berenstein, A. D.
and Zelevinsky, A. V.},
title={Triple Multiplicities for sl(r + 1) and the Spectrum of the Exterior Algebra of the Adjoint Representation},
journal={Journal of Algebraic Combinatorics},
year={1992},
month={05},
day={01},
volume={1},
number={1},
pages={7-22},
doi={10.1023/A:1022429213282},
url={https://doi.org/10.1023/A:1022429213282}
}

@book{Stanley2001, 
place={Cambridge}, 
series={Cambridge Studies in Advanced Mathematics}, 
title={Enumerative Combinatorics: Volume 2}, 
publisher={Cambridge University Press}, 
author={Stanley, Richard P}, 
year={2001}, 
collection={Cambridge Studies in Advanced Mathematics}}

@article{KT1999,
    author = "Knutson, Allen and Tao, Terence",
    title = "{The honeycomb model of $GL_n(\mathbb{C})$ tensor products I: Proof of the saturation conjecture}",
    doi = "10.1090/s0894-0347-99-00299-4",
    journal = "J. Am. Math. Soc.",
    volume = "12",
    number = "04",
    pages = "1055--1091",
    year = "1999"
}

@book{HK1999,
author = {Herman, Gabor T. and Kuba, Attila.},
address = {Boston},
booktitle = {Discrete tomography: foundations, algorithms, and applications},
language = {eng},
lccn = {99020171},
publisher = {Birkh\"{a}user},
title = {Discrete tomography : foundations, algorithms, and applications },
year = {1999},
}

@book{CD2006,
author = {Colbourn, Charles J. and Dinitz, Jeffrey H.},
title = {Handbook of Combinatorial Designs, Second Edition (Discrete Mathematics and Its Applications)},
year = {2006},
publisher = {Chapman \& Hall/CRC}
}

@article{PRV1967,
 URL = {http://www.jstor.org/stable/1970351},
 author = {K. R. Parthasarathy and R. Ranga Rao and V. S. Varadarajan},
 journal = {Annals of Mathematics},
 number = {3},
 pages = {383--429},
 title = {Representations of Complex Semi-Simple Lie Groups and Lie Algebras},
 volume = {85},
 year = {1967}
}

\end{document}